\documentclass[12pt, a4paper]{article}
\usepackage{hyperref}
\usepackage{amsmath,amsfonts,amsthm}
\usepackage[T1]{fontenc}
\usepackage{graphicx}

\theoremstyle{plain}
\newtheorem{theorem}{Theorem}
\newtheorem{lemma}[theorem]{Lemma}
\newtheorem{corollary}[theorem]{Corollary}
\newtheorem{proposition}[theorem]{Proposition}
\newtheorem{conjecture}[theorem]{Conjecture}
\newtheorem{definition}[theorem]{Definition}

\theoremstyle{definition}
\newtheorem{example}[theorem]{Example}
\newtheorem{remark}[theorem]{Remark}

\theoremstyle{remark}

\newcommand{\cB}{{\mathcal B}}

\newcommand{\cF}{{\mathcal F}}

\newcommand{\cI}{{\mathcal I}}

\newcommand{\cP}{{\mathcal P}}

\newcommand{\cX}{{\mathcal X}}

\newcommand{\bbE}{\mathbb{E}}

\newcommand{\bbP}{\mathbb{P}}

\newcommand{\bbR}{\mathbb{R}}

\newcommand{\bbZ}{\mathbb{Z}}

\newcommand{\eee}{{\rm e}}

\newcommand{\Var}{\mathop{\mathrm{Var}}\nolimits}

\DeclareMathOperator*{\argmax} {\arg\,max}
\makeatletter
\let \@fnsymbol\@arabic
\makeatother

\begin{document}

\title{Random tessellations associated with max-stable random fields}
\author{Cl\'ement Dombry\thanks{Universit\'e de Franche-Comt\'e, Laboratoire de Math\'ematiques de Besan\c con, UMR CNRS 6623, 16 route de Gray,
25030 Besan\c con cedex, France.  Email: \url{clement.dombry@univ-fcomte.fr}.  Webpage: \url{http://cdombry.perso.math.cnrs.fr}} \;  and \; Zakhar Kabluchko\footnote{Universit\"at M\"unster, Institut f\"ur Mathematische Statistik, Orl\'{e}ans-Ring 10, 48149 M\"unster, Germany. Email: \url{zakhar.kabluchko@uni-muenster.de}}
}
\maketitle
\date{}
\abstract{With any max-stable random process $\eta$ on $\mathcal{X}=\bbZ^d$ or $\bbR^d$,  we associate a random tessellation of the parameter space $\mathcal{X}$.
The construction relies on the Poisson point process representation of the max-stable process $\eta$ which is seen as the pointwise maximum of a random collection of functions $\Phi=\{\phi_i, i\geq 1\}$. The tessellation is constructed as follows: two points $x,y\in \mathcal{X}$ are in the same cell if and only if there exists a function $\phi\in\Phi$ that realizes
the maximum $\eta$ at both points $x$ and $y$, i.e.\ $\phi(x)=\eta(x)$ and $\phi(y)=\eta(y)$. We characterize the distribution of cells in terms of coverage and inclusion probabilities. Most interesting is the stationary case where the asymptotic properties of the cells are strongly related to the ergodic and mixing properties of the max-stable process $\eta$ and to its conservative/dissipative and positive/null decompositions.}

\vspace{0.5cm}


\section{Introduction}

Max-stable random fields provide popular and meaningful models for spatial extremes, see, e.g.,\ de Haan and Ferreira \cite{dHF06}. The reason is that they appear as the only possible non-degenerate limits for normalized pointwise maxima of independent and identically distributed random fields. The one-dimensional marginal distributions of max-stable fields belong to the parametric class of Generalized Extreme Value distributions. Being interested mostly in the dependence structure, we will restrict our attention to max-stable fields $\eta=(\eta(x))_{x\in\cX}$ on $\cX\subset\bbR^d$ with standard unit Fr\'echet margins, i.e.\ satisfying
\begin{equation}\label{eq:frechet_standard}
\bbP[\eta(x)\leq z]=\exp(-1/z)\quad \mbox{for all } x\in\cX \mbox{ and } z>0.
\end{equation}
The max-stability property has then the simple form
\[
n^{-1}\bigvee_{i=1}^n \eta_i\stackrel{d}{=}\eta \quad \mbox{for all } n\geq 1,
\]
where $(\eta_i)_{1\leq i\leq n}$ are i.i.d.\ copies of $\eta$, $\bigvee$ is the pointwise maximum, and $\stackrel{d}{=}$ denotes the equality of finite-dimensional distributions.

A fundamental tool in the study of max-stable processes is their spectral representation (see, e.g.,\ de Haan \cite{dH84}, Gin\'e {\it et al.}\ \cite{GHV90}):
any stochastically continuous max-stable process $\eta$ can be written as
\begin{equation}\label{eq:deHaan}
\eta(x)=\bigvee_{i\geq 1} U_iY_i(x),\quad x\in \mathcal{X},
\end{equation}
where
\begin{itemize}
\item[-] $(U_i)_{i\geq 1}$ is the decreasing enumeration of the points of a Poisson point process on $(0,+\infty)$ with intensity measure $u^{-2}\mathrm{d}u$,
\item[-] $(Y_i)_{i\geq 1}$ are i.i.d.\ copies of a non-negative stochastic process $Y$ on $ \mathcal{X}$ such that $\bbE[Y(x)]=1$ for all $x\in \mathcal{X}$,
\item[-] the sequences $(U_i)_{i\geq 1}$ and $(Y_i)_{i\geq 1}$ are independent.
\end{itemize}
In this paper, we focus on max-stable random fields defined on $\mathcal{X}=\bbZ^d$ or $\bbR^d$.  In the case $\mathcal{X}=\bbR^d$ we always assume that $\eta$ has continuous sample paths. Equivalently, the spectral process $Y$ has continuous sample paths and
\begin{equation}\label{eq:deHaan2}
\bbE\Big[\sup_{x\in K} Y(x)\Big]<\infty \mbox{ for every compact set } K\subset \bbR^d.
\end{equation}
Note that the equivalence follows for instance from de Haan and Ferreira \cite[Corollary 9.4.5]{dHF06}.

Representation \eqref{eq:deHaan} has a nice interpretation pointed out by Smith \cite{S90} and Schlather \cite{S02}. In the context of a rainfall model, we can interpret each index $i\geq 1$ as a {\it storm event}, where $U_i$ stands for the intensity of the storm and $Y_i$ stands for its shape; then $U_iY_i(x)$ represents the amount of precipitation  due to the storm event $i$ at
point $x\in\cX$, and $\eta(x)$ is the maximal precipitation over all storm events at this point. This interpretation raises a natural question: what is the shape of the region
$C_i\subset \cX$ where the storm $i$ is extremal?  More formally, we define the {\it cell} associated with the storm event $i\geq 1$  by
$C_i=\{x\in \mathcal{X};\ U_iY_i(x)=\eta(x)\}$. It is a (possibly empty) {\it random closed subset} of $\cX$ and each point $x\in\mathcal{X}$ belongs almost surely to a unique cell (the point process $\{U_iY_i(x)\}_{i\geq 1}$ is a Poisson point process with intensity $u^{-2}\mathrm{d}u$ so that the maximum $\eta(x)$ is almost surely attained for unique $i$).

A drawback of this approach is that the distribution of the cell $C_i$ depends on the specific representation \eqref{eq:deHaan}.
For instance, with the convention that the sequence $(U_i)_{i\geq 1}$ is decreasing,  the cell $C_1$ is stochastically larger than the other cells. To avoid this, we introduce a canonical way to define the tessellation.
\begin{definition}
For $x\in \mathcal{X}$, the cell of $x$ is the random closed subset
\begin{equation}\label{eq:defCx}
C(x)=\{y\in \mathcal{X};\ \exists i\geq 1,\ U_iY_i(x)=\eta(x)\ \mbox{and}\ U_iY_i(y)=\eta(y)\}.
\end{equation}
\end{definition}
\noindent
The cell $C(x)$ is non-empty since it contains $x$. In the case $\cX=\bbZ^d$, for any two points $x_1,x_2\in\bbZ^d$, the cells $C(x_1)$ and $C(x_2)$ are almost surely either equal or disjoint. In the case $\cX=\bbR^d$, for any two points $x_1,x_2\in\bbR^d$, the cells $C(x_1)$ and $C(x_2)$ are almost surely either equal or have disjoint interiors.

The purpose of this paper is to study some properties of the {\it random tessellation} $(C(x))_{x\in \mathcal{X}}$. Let us stress  that in this paper the terms {\it cell} and {\it tessellation} are meant in a broader sense than in stochastic geometry where they originated. Here, a cell is a general (not necessarily convex or connected) random closed set and a tessellation is a random covering of $\cX$ by closed sets with pairwise disjoint interiors. The following lemma provides a first simple but important observation.
\begin{lemma}\label{lem0}
The distribution of the tessellation $(C(x))_{x\in \mathcal{X}}$ depends on the distribution of the max-stable  process $\eta$ only and  not  on the specific representation \eqref{eq:deHaan}.
\end{lemma}
\noindent
To prove the lemma, introduce the functional point process (which will play a key role in the sequel)
\begin{equation}\label{eq:phi}
 \Phi=\{\phi_i,\ i\geq 1\} \quad \mbox{where}\ \phi_i=U_iY_i, \ i\geq 1.
\end{equation}
Note that $\phi_i$ are elements of $\cF_0=\cF(\mathcal{X},[0,+\infty))\setminus \{0\}$, the set of non-negative and continuous functions on $\mathcal{X}$ excluding the zero function.  (We may assume without loss of generality that $Y$ does not vanish identically). The set $\cF_0$ is endowed with the $\sigma$-algebra generated by the coordinate mappings.
It follows from the transformation theorem that $\Phi$ is a Poisson point process  on $\cF_0$ with intensity measure $\mu$ given by
\begin{equation}\label{eq:mu}
\mu(A)=\int_0^\infty \bbP[uY\in A]u^{-2}\mathrm{d}u,\quad A\subset\cF_0 \mbox{ Borel}.
\end{equation}
The measure $\mu$ is called  the exponent measure or max-L\'evy measure and is related to the multivariate cumulative distribution functions of $\eta$ by
\begin{eqnarray*}
\lefteqn{\bbP[\eta(x_j)\leq z_j, j=1,\ldots,n]}\\
&=&\exp\left(-\mu(\{f\in\cF_0; \ f(x_j)>z_j \mbox{ for some } j=1,\ldots,n\})\right)
\end{eqnarray*}
for all $n\geq 1$, $x_1,\ldots,x_n\in\mathcal{X}$ and $z_1,\ldots,z_n>0$. In particular, this shows that $\mu$  depends on the distribution of $\eta$ only
and does not depend on the specific representation \eqref{eq:deHaan}. Now, Lemma \ref{lem0} follows easily since the tessellation $(C(x))_{x\in \mathcal{X}}$ is a functional of the Poisson point process $\Phi$ with intensity $\mu$.

The aim of this paper is to study some properties of the tessellation $(C(x))_{x\in \mathcal{X}}$ and to relate them to the properties of the max-stable random field $(\eta(x))_{x\in\cX}$.
It is worth noting that some well-known tessellations like the Laguerre and  some Johnson--Mehl tessellations (see, e.g.,\ M{\o}ller \cite{mol1992}) are particular cases of this setting (see Examples \ref{ex:1} and \ref{ex:1bis} below). Furthermore, thanks to the Poisson point process representation  by Gin\'e {\it et al.} \cite{GHV90}, the results from the present paper could presumably be extended to the more general framework of upper semi-continuous max-infinitely divisible processes. The connection with stochastic geometry would even be  stronger via the notion of hypograph: the hypograph of an upper semi-continuous max-infinitely divisible process can be represented as the union of random closed sets from a Poisson point process.   However, for the sake of simplicity, we consider only the case of continuous max-stable processes for which more results are available from the literature.

The paper is structured as follows. In Section 2, we study the law of the cell $C(x)$ and provide some formulas for the inclusion and coverage probabilities as well as some examples.
In Section 3, we focus on the stationary case and establish strong connections between asymptotic properties of $C(0)$ and properties of the max-stable random field $\eta$ such as ergodicity, mixing and decompositions of the  non-singular flow associated with $\eta$. Theorem \ref{theo2} relates the boundedness of the cell to the conservative/dissipative decomposition.
Theorem \ref{theo3} links the asymptotic density of the cell with the positive/null decomposition. Proofs are collected in Sections 4 and 5.

\section{Basic properties and examples}
\subsection{Basic properties}
Our first result is a simple characterization of the distribution of the cells of the tessellation.
\begin{theorem}\label{theo1}
Consider a sample continuous max-stable random field $\eta$ given by representation \eqref{eq:deHaan}.
For every $x\in \mathcal{X}$ and every measurable set $K\subset \mathcal{X}$,
\begin{equation}\label{eq:theo1.1}
\bbP[K\subset C(x)]=\bbE\left[ \inf_{y\in K\cup \{x\}}\frac{Y(y)}{\eta(y)}\right]
\end{equation}
and
\begin{equation}\label{eq:theo1.2}
\bbP[C(x)\subset K]=\bbE\left[ \left( \frac{Y(x)}{\eta(x)} - \sup_{y\in K^c}\frac{Y(y)}{\eta(y)}\right)^+\right],
\end{equation}
where $Y$ is independent of $\eta$, $K^c= \mathcal{X}\setminus K$ is the complement of the  set $K$, and $(z)^+=\max(z,0)$ is the positive part of $z$.
\end{theorem}
It is well-known that the distribution of a random closed set $C\subset \mathcal{X}$ is completely determined by its capacity functional
\[
\mathcal{X}_{C}(K)=\bbP[C\cap K\neq \emptyset],\quad K\subset \mathcal{X}\ \mbox{compact},
\]
see, e.g.,\ Molchanov \cite[Chapter 1]{M05}. Clearly, Theorem~\ref{theo1} implies that for all $x\in\mathcal{X}$ the capacity functional of the cell $C(x)$ is given by
\begin{eqnarray*}
\mathcal{X}_{C(x)}(K)
&=&1-\bbE\left[ \left( \frac{Y(x)}{\eta(x)} - \sup_{y\in K}\frac{Y(y)}{\eta(y)}\right)^+\right].
\end{eqnarray*}

\begin{remark}{\rm
It is worth noting that Weintraub \cite{W91} introduced  (with a different terminology) the probability that two points $x$ and $y$ are in the same cell
as a measure of dependence between $\eta(x)$ and $\eta(y)$.  More precisely, he considered
\[
\pi(x,y)=\bbP[y\in C(x)]=\bbE\left[\frac{Y(x)}{\eta(x)}\wedge \frac{Y(y)}{\eta(y)}\right],\quad x,y\in \mathcal{X}.
\]
Clearly, $\pi(x,y)\in [0,1]$.
One can prove easily that $\pi(x,y)=0$ holds if and only if $\eta(x)$ and $\eta(y)$ are independent, while $\pi(x,y)=1$ if and only if $\eta(x)=\eta(y)$ almost surely.
Moreover, $\pi(x,y)$ can be compared to  the extremal coefficient $\theta(x,y)$ which is another well-known measure of dependence for max-stable processes defined by
\begin{equation}\label{eq:theta}
\theta(x,y)=-\log \bbP[\eta(x)\vee\eta(y)\leq 1] \in [1,2].
\end{equation}
According to Stoev \cite[Proposition 5.1]{S08}, we have
\begin{equation}\label{eq:comparison}
\frac{1}{2}(2-\theta(x,y))\leq \pi(x,y)\leq 2(2-\theta(x,y)).
\end{equation}
In the case of stationary max-stable random fields, we use the notation $\theta(h)=\theta(0,h)$ and $\pi(h)=\pi(0,h)$.
}
\end{remark}

As a by-product of Theorem \ref{theo1}, we can provide an explicit expression for the mean volume of the cells.
Denote by $\lambda$ the discrete counting measure when $\mathcal{X}=\bbZ^d$ or the Lebesgue measure when $\mathcal{X}=\bbR^d$. The volume of $C(x)$ is defined by
$\mathrm{Vol}(C(x))=\lambda(C(x))$. In the discrete case, $\mathrm{Vol}(C(x))$ is the cardinality of $C(x)$.
\begin{corollary}\label{cor1.1}
Let $x\in \mathcal{X}$. The cell $C(x)$ has expected volume
\[
\bbE[\mathrm{Vol}(C(x))]=\int_\mathcal{X} \bbE\left[ \frac{Y(x)}{\eta(x)}\wedge \frac{Y(y)}{\eta(y)}\right]\lambda(\mathrm{d}y).
\]
\end{corollary}
\noindent
This together with \eqref{eq:comparison} implies that the cell $C(x)$ has finite expected volume if and only if
$\int_\mathcal{X} (2-\theta(x,y))\lambda(\mathrm{d}y)<+\infty$. Another consequence of Theorem \ref{theo1} is an expression
for the probability that the cell $C(x)$ is bounded.
\begin{corollary}\label{cor1.2}
Let $x\in \mathcal{X}$. The cell $C(x)$ is bounded with probability
\[
\bbP[C(x)\ \mbox{is bounded}]=\bbE\left[ \left( \frac{Y(x)}{\eta(x)} - \limsup_{y\to \infty}\frac{Y(y)}{\eta(y)}\right)^+\right].
\]
Furthermore, the following statements are equivalent:
\begin{enumerate}
\item[i)] the cell $C(x)$ is bounded a.s.;
\item[ii)]  $\lim_{y\to\infty}\frac{Y(y)}{\eta(y)}= 0$ a.e.\ on the event $\{Y(x)\neq 0\}$.
\end{enumerate}
\end{corollary}
\begin{remark}
In the case when the max-stable process $\eta$ is stationary, we will see in Section \ref{subsec:boundedness} below that condition ii) can be replaced by the following one: $Y(y)\to 0$ a.s.\ as $y\to \infty$.
\end{remark}

\subsection{Examples}
As an illustration and to get some intuition, we provide several examples. Simulations of the max-stable processes together with the associated tessellations are available
on the personal webpage of the first author.

\begin{example}\label{ex:1}{\rm
The isotropic Smith process \cite{S90} is defined by
\[
 \eta(x)=\bigvee_{i\geq 1} U_ih(x-X_i),\quad x\in\bbR^d,
\]
where $\{(U_i,X_i),i\geq 1\}$ is a Poisson point process on $(0,\infty)\times\bbR^d$ with intensity $u^{-2}\mathrm{d}u\mathrm{d}x$
and $h(x)=(2\pi)^{-d/2}\exp(-\|x\|^2/2)$ is the standard Gaussian $d$-variate density function. The Smith process is a stationary max-stable process that belongs to the class of {\it moving maximum processes}
and is hence mixing. Surprisingly, the associated tessellation is exactly the so-called {\it Laguerre tessellation} studied in great detail by Lautensack and Zuyev \cite{Z08}. Indeed, the cell $C_i$ is given by
\[
 C_i=\{x\in\bbR^d;\ \|x-X_i\|^2-2\ln(U_i) \leq \|x-X_j\|^2-2\ln(U_j),\ \ j\neq i \}.
\]
In this very specific example, the cells are convex bounded polygons.
}\end{example}

\begin{example}\label{ex:1bis}{\rm
Consider a moving maximum process of the same form as in the previous example, but with $h(x) = \exp\{-\|x\|/v\}$, $x\in\bbR^d$, where $v>0$ is a parameter.  Then, the cell $C_i$ is given by
\[
C_i=\{x\in\bbR^d;\ \|x-X_i\|/v - \ln(U_i) \leq \|x-X_j\|/v - \ln(U_j),\ \ j\neq i \},
\]
and we recover a special case of the \textit{Johnson--Mehl tessellation}; see M{\o}ller \cite{mol1992}. }
\end{example}

\begin{example}\label{ex:2}{\rm
The stationary extremal Gaussian process originally introduced by Schlather \cite{S02} corresponds to the case when the spectral process $Y$ in representation \eqref{eq:deHaan} is given by
\[
 Y(x)=\sqrt{\frac{\pi}{2}}\max(W(x),0),\quad x\in\mathbb{R}^d,
\]
where $W$ is a stationary Gaussian process on $\mathbb{R}^d$ with zero mean, unit variance and correlation function $\rho(h)=\bbE[W(0)W(h)]$, $h\in\mathbb{R}^d$. The  extremal coefficient function is given by
\[
 \theta(h)=2T_2\left[\sqrt{\frac{2}{1-\rho(h)^2}}-\sqrt{\frac{1-\rho(h)^2}{2}}\rho(h)\right],\quad h\in\mathbb{R}^d,
\]
where $T_2$ is the cumulative distribution function of a Student distribution with $2$ degrees of freedom. Typically, $\rho(h)\to 0$ as $h\to\infty$, so that  $\theta(h)\to 2T_2(\sqrt{2})<2$ and  $\eta$ is neither mixing nor ergodic (see Stoev \cite{S08} or Kabluchko and Schlather \cite{KS10}). The  inequalities in \eqref{eq:comparison} entail that
$\liminf \bbP[h\in C(0)]>0$ as $h\to\infty$. This suggests that the cells are not bounded which is consistent with the simulations available on the first author personal webpage. Note that the cells are neither convex nor connected and have a  pretty regular shape due to the particular choice of the correlation function $\rho(h)=\exp(-\|h\|^2/2)$ that yields smooth Gaussian sample paths.
}\end{example}

\begin{example}\label{ex:3}{\rm
Brown--Resnick processes \cite{KSdH09} form a flexible class of max-stable processes. They are given by  \eqref{eq:deHaan} with the spectral process of the form
 \[
  Y(x)=\exp\left(W(x) - \frac 12 \sigma^2(x)\right),\quad x\in\mathbb{R}^d,
 \]
where $W$ is a stationary increment centered Gaussian process on $\mathbb{R}^d$, and $\sigma^2(x)=\Var W(x)$.
Surprisingly, the process $\eta$ is stationary \cite{KSdH09}. Its distribution is completely characterized by the variogram
\[
\gamma(h) = \Var(W(x+h)-W(x)), \quad h\in\mathbb{R}^d.
\]
The  extremal coefficient function is given by (\cite[p.~2063]{KSdH09})
\[
 \theta(h)=2 G\left(\frac 12 \sqrt{\gamma(h)}\right),\quad h\in\mathbb{R}^d,
\]
where $G$ is the cumulative distribution function of the standard normal distribution. Typically, $\gamma(h)\to\infty$ as $h\to\infty$, so that  $\theta(h)\to 2$ and  $\eta$ is mixing \cite{S08,KS10}.
The inequalities in \eqref{eq:comparison} entail that $\lim \bbP[h\in C(0)]=0$ as $h\to\infty$ suggesting that the cells become asymptotically independent at large distances.
Since $1-G(u)\sim 1/(\sqrt{2\pi} u)\, \eee^{-u^2/2}$, $u\to+\infty$, Corollary \ref{cor1.1} implies that  the cell $C(0)$ has finite expected volume (and hence, is a.s.\ bounded) provided that the following condition is satisfied:
$$
\liminf_{h\to\infty} \frac{\gamma(h)}{\log \|h\|} > 8d.
$$
Simulations with the variogram $\gamma(h)=2\|h\|$ are available on the first author's personal webpage. They show that the cells may have a very rough shape, due to the particular choice of the variogram that yields rough Gaussian paths.
}\end{example}

\section{The stationary case:  asymptotic properties of cells}\label{sec:asy}
In the sequel, we focus on the case when $\eta$ is a \textit{stationary} sample continuous max-stable random field on $ \mathcal{X}=\bbZ^d$ or $\bbR^d$.
We show strong connections between the ergodic and mixing properties of the random field $\eta$, its conservative/dissipative and positive/null decompositions and
the geometry of the cells.

\subsection{Ergodic properties and geometry of the cells}\label{sec:CEP}
Ergodic and mixing properties of max-stable random fields have been studied intensively by Stoev \cite{S08,S10} and Kabluchko and Schlather \cite{KS10}.
A simple characterization using the extremal coefficient is known (see, e.g.,\ \cite[Theorems 1.1 and 1.2]{KS10} where the more general case of max-infinitely divisible processes is considered).
\begin{theorem}[Stoev (2008), Kabluchko and Schlather (2010)]\label{thm:SKS}\ \\
Let $\eta$ be a stationary max-stable random field on $\cX=\mathbb{Z}^d$ or $\mathbb{R}^d$.
\begin{itemize}
 \item[-] $\eta$ is ergodic if and only if $\theta(h)\to 2$ in Ces\`{a}ro mean as $h\to\infty$;
 \item[-] $\eta$ is mixing if and only if $\theta(h)\to 2$  as $h\to\infty$.
\end{itemize}
\end{theorem}
Interestingly, these results can be reinterpreted in terms of the geometric properties of the tessellation. For $r>0$, we write $B_r=[-r,r]^d\cap\cX$. We equip $ \mathcal{X}$ with a measure $\lambda$ which is either the counting or the Lebesgue measure, when  $\mathcal{X}=\bbZ^d$ or $\mathcal{X}=\bbR^d$, respectively.
\begin{proposition}\label{prop1}
Let $\eta$ be a stationary, sample continuous max-stable random field on $\mathcal{X}=\bbZ^d$ or $\bbR^d$.
\begin{enumerate}
\item The following statements are equivalent:
\begin{enumerate}
\item[(1.a)] $\eta$ is ergodic,
\item[(1.b)] $\lim_{r\to+\infty} \bbE\Big[\frac{\lambda(C(0)\cap B_r)}{\lambda(B_r)}\Big]=0$.
\end{enumerate}
\item The following statements are equivalent:
\begin{enumerate}
\item[(2.a)] $\eta$ is mixing,
\item[(2.b)] $\lim_{x\to\infty} \bbP[x\in C(0)]=0$.
\end{enumerate}
\end{enumerate}
\end{proposition}

Next we focus on strong mixing properties of max-stable processes, see Dombry and Eyi-Minko \cite{DEM12}.
The $\beta$-mixing coefficients of the random process $\eta$ are defined as follows: for disjoint closed subsets $S_1,S_2\subset \mathcal{X}$,
we define
\begin{eqnarray}\label{eq:beta}
\beta(S_1,S_2)= \sup\Big\{ |\cP_{S_1\cup S_2}(C)- (\cP_{S_1}\otimes \cP_{S_2})(C)|;\ C\in \cB_{S_1\cup S_2}\Big\},
\end{eqnarray}
where $\cP_S$ is the distribution (on the space $\bbR_+^S$) of the restriction of $\eta$ to the set $S$, and $\cB_S$ is the product $\sigma$-algebra on the space $\bbR_+^S$.  Given a closed subset $S\subset \mathcal{X}$ and $r>0$, we define
\[
\beta_r(S)=\beta(S,S_r^c)\mbox{ with }  S_r^c=\{x\in S; d(x,S)\geq r\}
\]
where $d(x,S)$ denotes the distance between the point $x$ and the set $S$.
We say that $\eta$ is strongly $\beta$-mixing if for all compact sets $S\subset \mathcal{X}$,
\[
\lim_{r\to+\infty} \beta_r(S)=0.
\]
\begin{proposition}\label{prop2}
If $\eta$ is a stationary max-stable random field such that $C(0)$ is  almost surely bounded, then  $\eta$ is strongly $\beta$-mixing.
\end{proposition}
We conjecture that the converse implication is also true:
\begin{conjecture}
If $\eta$ is a strongly $\beta$-mixing stationary max-stable random field, then $C(0)$ is  almost surely bounded.
\end{conjecture}
We were not able to prove the conjecture, mainly because we lack a lower bound for the $\beta$-mixing coefficient $\beta(S_1,S_2)$ (only an upper bound is given in \cite{DEM12}).
The intuition is relatively clear though: if the cell remains unbounded with positive probability, then the value $\eta(0)$ of the random field at the origin may have an impact at infinity
via the unbounded cell $C(0)$. In view of Corollary \ref{cor2}, see below, the conjecture can also be stated as follows: a stationary max-stable field $\eta$ is strongly $\beta$-mixing
 if and only if $\eta$ is purely dissipative.

\subsection{Preliminaries on stationary max-stable processes}

The structure of stationary max-stable processes was first investigated by de Haan and Pickands \cite{dHP86}.
Recently, further results were obtained by exploiting the analogy between the theory of max-stable and sum-stable processes.
Inspired by the works of Rosi\'nski \cite{R95,R00}, Rosi\'nski and Samorodnitsky \cite{RS96} and Samorodnitsky \cite{S04,S05},
the representation theory of stationary max-stable random fields via non-singular flows was developed independently by Kabluchko \cite{K09}, Wang and  Stoev \cite{WS10} and Wang {\it et al.}\ \cite{WRS13}.
See also Kabluchko and Stoev \cite{KS12} for an extension to sum- and max-infinitely divisible processes.
In these works, the conservative/dissipative and positive/null decompositions of the non-singular flow play a major role.

To avoid technical details of non-singular ergodic theory, we use a naive approach based on cone decompositions of max-stable processes (see, e.g.,\  Wang and Stoev \cite[Theorem 5.2]{WS10}).
The links between this approach and the non-singular ergodic theory are explored in Dombry and Kabluchko \cite{DK15}.

The following simple lemma about cone decompositions of max-stable processes will be useful.
Recall that $\cF_0=\cF( \mathcal{X},[0,+\infty))\setminus\{0\}$ denotes the set of continuous, non-negative functions on $\mathcal{X}$ excluding the zero function.
A measurable subset $\mathcal{C}\subset \mathcal{F}_0$ is called a {\it cone} if for all $f\in\mathcal{C}$ and $u>0$, $uf\in\mathcal{C}$.
The cone $\mathcal{C}$ is said to be {\it shift-invariant} if for all $f\in\mathcal{C}$ and $x\in\mathcal{X}$, we have $f(\cdot+x)\in\mathcal{C}$.
\begin{lemma}\label{lem:dec}
Let $\mathcal{C}_1$ and $\mathcal{C}_2$ be two measurable, shift-invariant cones such that $\mathcal{F}_0=\mathcal{C}_1\cup\mathcal{C}_2$ and $\mathcal{C}_1\cap\mathcal{C}_2=\emptyset$.
Let $\eta$ be a stationary max-stable process given by representation \eqref{eq:deHaan}. Consider the decomposition $\eta=\eta_1\vee \eta_2$ with
\[
\eta_1(x)=\bigvee_{i\geq 1} U_iY_i(x)1_{\{Y_i\in\mathcal{C}_1\}}\quad \mbox{and}\quad\eta_2(x)=\bigvee_{i\geq 1} U_iY_i(x)1_{\{Y_i\in\mathcal{C}_2\}}.
\]
Then, $\eta_1$ and $\eta_2$ are stationary and independent max-stable processes\footnote{with margins differing from the standardized form~\eqref{eq:frechet_standard} by a multiplicative constant} whose  distribution depends only on the distribution of $\eta$ and not on the specific representation \eqref{eq:deHaan}.
\end{lemma}
The notion of Brown--Resnick stationarity introduced in Kabluchko {\it et al.}\ \cite{KSdH09} will be useful.
\begin{definition}\label{def:BR}
We say that the process $Y=(Y(x))_{x\in\cX}$ is {\it Brown--Resnick stationary} if the associated max-stable process $\eta$ defined by \eqref{eq:deHaan} is stationary.
\end{definition}
For future reference, we gather in the next lemma several properties of Brown-Resnick stationary processes. A shift-invariant cone $\cF_L$ is said to be {\it localizable} if there exist mappings $L_1:\cF_L\to \mathcal{X}$ and $L_2:\cF_L\to (0,+\infty)$
such that for all $f\in \cF_L$, $x\in \mathcal{X}$ and $u>0$,
\begin{itemize}
\item[-] $L_1(f(\cdot+x))=L_1(f)-x$ and $L_1(uf)=L_1(f)$,
\item[-] $L_2(f(\cdot+x))=L_2(f)$ and $L_2(uf)=uL_2(f)$.
\end{itemize}
\noindent
A typical example of localizable cone is the cone $\left\{f\in\cF_0;\lim_\infty f=0\right\}$ with $L_1(f)=\argmax f$ and $L_2(f)=\max f$ (if the maximum is attained at several points, we define the $\argmax$ as the smallest such point with respect to the lexicographic order).
\begin{lemma}\label{lemBR}
Let $Y$ and $Y'$ be independent Brown--Resnick stationary processes. In the case $\cX=\bbR^d$, we assume for statements iii) and iv) that the associated max-stable process has continuous sample paths.
\begin{enumerate}
\item[i)] The product $YY'$ is also Brown--Resnick stationary.
\item[ii)] Let $C$ be a shift-invariant cone, then $Y1_{\{Y\in C\}}$ is Brown--Resnick stationary.
\item[iii)] Let $K\subset\cX$ be  compact. In the case $\cX=\bbR^d$,  we suppose that the interior  of $K$ is non-empty. Then, modulo null sets,
\[
 \left\{\lim_{x\to\infty} Y(x)=0\right\}=\left\{\int_{\mathcal{X}}\sup_{y\in K}Y(x+y)\lambda(\mathrm{d}x)<\infty\right\}.
\]
\item[iv)] 
The cone $\cF_L=\{f\in\cF_0; \sup_{x\in\cX} f(x)> \limsup_{x\to\infty} f(x)\}$ is localizable and, modulo null sets,
\[
\{Y\in\cF_L\}\subset  \left\{\lim_{x\to\infty} Y(x)=0\right\}.
\]
In fact, the latter inclusion holds for any localizable cone.  
\end{enumerate}
\end{lemma}
Statement i) is due to Kabluchko {\it et al.}\ \cite[Corollary 8]{KSdH09}, statement ii) is a by-product of Lemma \ref{lem:dec} and its proof.
Statements iii) and iv) are closely related to Proposition 10 and its proof in Kabluchko and Dombry \cite{DK15}. In the proof of \cite[Proposition 10]{DK15}, we show that $\mathcal{F}_D=\mathcal{F}_D'=\tilde\cF_D$ wich implies iii); statement iv) is proven with the same arguments as for the proof of $\tilde\cF_D\subset \cF_D$ in \cite{DK15}, see the appendix for more details.

\subsection{Boundedness of cells}\label{subsec:boundedness}
We prove that the boundedness of the cell $C(x)$, $x\in\mathcal{X}$, is strongly connected with the conservative/dissipative decomposition of the max-stable process $\eta$.
Introduce the following shift-invariant cones of functions:
\begin{eqnarray}
\cF_C&=&\left\{f\in \cF_0;\ \limsup_{x\to\infty} f(x)>0 \right\},\label{eq:CC}\\
\cF_D&=&\left\{f\in \cF_0;\ \lim_{x\to\infty} f(x)=0 \right\}\label{eq:CD}.
\end{eqnarray}
The conservative/dissipative decomposition  of $\eta$ is given by
\begin{eqnarray}
\eta_C(x)&=& \bigvee_{i\geq 1} U_iY_i(x)1_{\{Y_i\in\cF_C\}},\label{eq:cons}\\
\eta_D(x)&=& \bigvee_{i\geq 1} U_iY_i(x)1_{\{Y_i\in\cF_D\}}.\label{eq:dis}
\end{eqnarray}
According to Lemma \ref{lem:dec}, the processes $\eta_C$ and $\eta_D$ are independent stationary max-stable processes such that $\eta=\eta_C\vee \eta_D$.
In dimension $d=1$, this cone decomposition is known to be related to the  the conservative/dissipative decomposition of the non-singular flow generating $\eta$ (see, e.g.,\ Wang and Stoev \cite[Theorem 5.2]{WS10}).
The following theorem relates this conservative/dissipative decomposition to the boundedness of the cell $C(x)$.
\begin{theorem}\label{theo2}
Let $x\in \mathcal{X}$. The following events are equal modulo null sets:
\begin{equation}\label{theo2.1}
\{C(x)\ \mbox{is unbounded}\}=\{\eta_C(x)>\eta_D(x)\},
\end{equation}
\begin{equation}\label{theo2.2}
\{C(x)\ \mbox{is bounded}\}=\{\eta_D(x)>\eta_C(x)\}.
\end{equation}
\end{theorem}

We denote by $\alpha_C$ and $\alpha_D$ the scale parameters of the $1$-Fr\'echet random variables $\eta_C(x)$ and $\eta_D(x)$ respectively, i.e.\ for all $z>0$,
\begin{equation}\label{eq:defalpha}
\bbP[\eta_C(x)\leq z ]=\exp(-\alpha_C/z),\quad  \bbP[\eta_D(x)\leq z ]=\exp(-\alpha_D/z).
\end{equation}
Note that $\alpha_D+\alpha_C=1$ and that $\alpha_C$ and $\alpha_D$ do not depend on $x\in\cX$.
We say that $\eta$ is purely conservative (resp.\ purely dissipative) if $\alpha_C=1$ (resp. $\alpha_D=1$).

\begin{corollary}\label{cor2} Let $x\in\mathcal{X}$. We have:
\begin{itemize}
\item[i)] $\bbP[C(x)\ \mbox{is unbounded}]=\alpha_C$,
\item[ii)] $\bbP[C(x)\ \mbox{is bounded}]=\alpha_D$,
\item[iii)] $C(x)$ is unbounded a.s.\ if and only if $\eta$ is purely conservative,
\item[iv)] $C(x)$ is bounded a.s.\ if and only if $\eta$ is purely dissipative.
\end{itemize}
\end{corollary}

\subsection{Asymptotic density of cells}
Next we consider the decomposition of $\eta$ into positive and null  components and relate it to the asymptotic density of the cell $C(x)$.
For this purpose, we introduce a new construction of the positive/null decomposition of max-stable processes which simplifies and extends to the dimension $d\geq 1$ the construction from Samorodnitsky \cite{S05} and Wang and Stoev \cite[Example 5.4]{WS10}.

Recall that we write $B_r=[-r,r]^d\cap\cX$ for $r>0$ and that $\lambda$   is either the counting or the Lebesgue measure on $ \mathcal{X}$, when  $\mathcal{X}=\bbZ^d$ or $\mathcal{X}=\bbR^d$, respectively. Consider the shift-invariant cones of functions
\begin{eqnarray}
\cF_P&=&\left\{f\in \cF_0;\ \lim_{r\to\infty}\frac{1}{\lambda(B_r)}\int_{B_r} f(x)\lambda(\mathrm{d}x)>0\right\},\label{eq:FP}\\
\cF_N&=&\left\{f\in \cF_0;\ \liminf_{r\to\infty}\frac{1}{\lambda(B_r)}\int_{B_r} f(x)\lambda(\mathrm{d}x)=0\right\}.\label{eq:FN}
\end{eqnarray}
In the definition of $\cF_P$, we assume that the limit exists.  The stationarity of $\eta$ implies that $Y\in \cF_P\cup \cF_N$ a.s.;\ see Dombry and Kabluchko~\cite{DK15}.
According to Lemma \ref{lem:dec}, the corresponding decomposition  is
\begin{eqnarray}
\eta_P(x)&=& \bigvee_{i\geq 1} U_iY_i(x)1_{\{Y_i\in\cF_P\}},\label{eq:pos}\\
\eta_N(x)&=& \bigvee_{i\geq 1} U_iY_i(x)1_{\{Y_i\in\cF_N\}},\label{eq:null}
\end{eqnarray}
where the processes $\eta_N$ and $\eta_P$ are independent, stationary, max-stable and $\eta=\eta_P\vee \eta_N$.
This decomposition based on cones is equal to the positive/null decomposition based on the underlying non-singular flow (see e.g.,\ Wang and Stoev \cite[Theorem 5.3]{WS10} in dimension $d=1$, Wang {\it et al.} \cite{WRS13} in dimension $d\geq 1$, Dombry and Kabluchko \cite{DK15}).

Given a measurable subset $C\subset \mathcal{X}$, we define its lower and upper asymptotic densities  by
\[
\delta^-(C)=\liminf_{r\to +\infty}\frac{\lambda(C\cap B_r)}{\lambda(B_r)},\quad \delta^+(C)=\limsup_{r\to +\infty}\frac{\lambda(C\cap B_r)}{\lambda(B_r)}.
\]
If $\delta^-(C)=\delta^+(C)$, the common value is called the asymptotic density of $C$ and denoted by $\delta(C)$.
The following theorem relates the positive/null decomposition of $\eta$ to the asymptotic density of the cell $C(x)$.
\begin{theorem}\label{theo3}
Let $x\in \mathcal{X}$. The following events are equal modulo null sets:
\begin{equation}\label{theo3.1}
\{\delta(C(x))>0\}=\{\eta_P(x)>\eta_N(x)\},
\end{equation}
\begin{equation}\label{theo3.2}
\{\delta^-(C(x))=0\}=\{\eta_N(x)>\eta_P(x)\},
\end{equation}
where the notation  $\delta(C(x))>0$ means that the asymptotic density $\delta(C(x))$ exists and is positive.
\end{theorem}

We denote by $\alpha_P$ and $\alpha_N$ the scale parameters of the $1$-Fr\'echet random variables $\eta_P(x)$ and $\eta_N(x)$ respectively, i.e.\ for all $z>0$,
\[
\bbP[\eta_P(x)\leq z]=\exp(-\alpha_P/z)\ \ \mbox{and}\ \ \bbP[\eta_N(x)\leq z]=\exp(-\alpha_N/z).
\]
Note that $\alpha_P+\alpha_N=1$ and that $\alpha_P$ and $\alpha_N$ do not depend on $x$.
We say that the max-stable process $\eta$ is generated by a positive (resp.\ null) flow if  $\alpha_P=1$ (resp.\ $\alpha_N=1$).

\begin{corollary}\label{cor3} Let $x\in\mathcal{X}$. We have:
\begin{itemize}
\item[i)] $\bbP[\delta(C(x))>0]= \alpha_P$,
\item[ii)] $\bbP[\delta^-(C(x))=0]=\alpha_N$,
\item[iii)] $\delta(C(x))>0$ a.s.\ if and only if $\eta$ is generated by a positive flow,
\item[iv)] $\delta^-(C(x))=0$  a.s.\ if and only if $\eta$ is generated by a null flow.
\end{itemize}
\end{corollary}

\section{Proofs related to Section 2}

\subsection{Proof of Theorem~\ref{theo1}}
\begin{proof}[Proof of Theorem~\ref{theo1}]
We first prove  \eqref{eq:theo1.1}. For $f,g:\mathcal{X}\to \bbR$  and $K\subset \mathcal{X}$, we use the notation
\[
f>_K g \quad \mbox{if and only if}\quad f(x)>g(x) \mbox{ for all }x\in K.
\]
For $i\geq 1$, we write $m_i=\bigvee_{j\neq i} \phi_j$ where $\phi_i=U_iY_i$ is defined by \eqref{eq:phi}.
Fix some $x\in\mathcal{X}$. Note that  $x\in C_i$ if and only if $\phi_i(x)\geq m_i(x)$, whence (modulo null sets)
\begin{eqnarray*}
\{K\subset C(x)\}&=&\{\exists i\geq 1,\ \phi_i(x)>m_i(x) \mbox{ and } \forall y\in K,\ \phi_i(y)>m_i(y)\}\\
&=&\{\exists i\geq 1,\  \phi_i>_{K\cup\{x\}}m_i\}.
\end{eqnarray*}
The events $\{\phi_i>_{K\cup\{x\}}m_i\}$, $i\geq 1$, are pairwise disjoint so that
\[
1_{\{K\subset C(x)\}}=\sum_{i\geq 1}1_{\{\phi_i>_{K\cup\{x\}}m_i\}}\quad  \mbox{a.s.}
\]
Hence, we obtain
\[
\bbP[K\subset C(x)]=\bbE\Big[\sum_{i\geq 1}1_{\{\phi_i>_{K\cup\{x\}}m_i\}}\Big].
\]
This expectation can be computed thanks to the Slivnyak--Mecke  formula (see, e.g.,\ Schneider and Weil~\cite[page~68]{schw08}). Recall from~\eqref{eq:phi} and~\eqref{eq:mu} that $\Phi=\{\phi_i, i\geq 1\}$ is a Poisson point process with intensity $\mu$ and
that $m_i$ is a functional of $\Phi\setminus \{\phi_i\}$. The Slivnyak--Mecke formula implies that
\begin{eqnarray*}
\bbP[K\subset C(x)]&=&\int_{\cF_0} \bbE\left[1_{\{f>_{K\cup\{x\}}\eta\}}\right]\mu(\mathrm{d}f).
\end{eqnarray*}
Using  \eqref{eq:mu}, we compute
\begin{eqnarray*}
 \int_{\cF_0} \bbE\left[1_{\{f>_{K\cup\{x\}}\eta\}}\right]\mu(\mathrm{d}f)
&=& \int_0^\infty \bbE\left[1_{\{uY>_{K\cup\{x\}}\eta\}}\right]u^{-2}\mathrm{d}u\\
&=& \bbE\left[\int_0^{\infty} 1_{\{u>\sup_{K\cup\{x\}}\eta/Y\}}u^{-2}\mathrm{d}u\right]\\
&=& \bbE\left[\inf_{K\cup \{x\}} Y/\eta\right].
\end{eqnarray*}
This proves  \eqref{eq:theo1.1}.

In the same spirit as in the proof of \eqref{eq:theo1.1}, we have
\[
\{C(x)\subset K\}=\{\exists i\geq 1,\ \phi_i(x)>m_i(x) \mbox{ and }  \phi_i<_{K^c}m_i\}
\]
whence we deduce
\[
\bbP[C(x)\subset K]= \bbE\Big[\sum_{i\geq 1}1_{\{\phi_i(x)>m_i(x)\}}1_{\{\phi_i<_{K^c}m_i\}}\Big].
\]
We obtain  \eqref{eq:theo1.2} thanks to the Slivnyak--Mecke and straightforward computations.
\end{proof}

\subsection{Proof of Corollaries~\ref{cor1.1} and~\ref{cor1.2} }
\begin{proof}[Proof of Corollary~\ref{cor1.1}]
By Fubini's Theorem, the expected volume of the cell $C(x)$ is equal to
\[
\bbE[\mathrm{Vol}(C(x))]=\bbE\left[\int_{\mathcal{X}}1_{\{y\in C(x)\}}\lambda(\mathrm{d}y)\right]=\int_{\mathcal{X}}\bbP[y\in C(x)]\lambda(\mathrm{d}y)
\]
and, according to Theorem \ref{theo1},
\[
\bbP[y\in C(x)]=\bbE\left[\frac{Y(x)}{\eta(x)}\wedge \frac{Y(y)}{\eta(y)}\right].
\]
\end{proof}

\begin{proof}[Proof of Corollary~\ref{cor1.2}]
For $n\geq 1$, we recall that $B_n=[-n,n]^d\cap\cX$. The sequence of events $\{C(x)\subset B_n\}$, $n\geq 1$, is non-decreasing and we have
\[
\{C(x)\mbox{  bounded}\}=\bigcup_{n\geq 1}\{C(x)\subset B_n\},
\]
whence
\[
\bbP[ C(x)\mbox{ bounded}]=\lim_{n\to\infty} \bbP[C(x)\subset B_n].
\]
Using  \eqref{eq:theo1.2}, we get
\[
\bbP[C(x)\subset B_n]=\bbE\left[\left(\frac{Y(x)}{\eta(x)}- \sup_{B_n^c}\frac{Y}{\eta}\right)^+\right].
\]
As $n\to+\infty$, the sequence $\sup_{B_n^c}Y/\eta$ decreases to $\limsup_{y\to\infty}Y(y)/\eta(y)$. The monotone convergence theorem entails that
\[
\lim_{n\to\infty}\bbE\left[\left(\frac{Y(x)}{\eta(x)}- \sup_{B_n^c}\frac{Y}{\eta}\right)^+\right]=\bbE\left[\left(\frac{Y(x)}{\eta(x)}- \limsup_{y\to\infty}\frac{Y(y)}{\eta(y)}\right)^+\right],
\]
whence we deduce
\[
\bbP[ C(x)\mbox{ bounded}]=\bbE\left[\left(\frac{Y(x)}{\eta(x)}- \limsup_{y\to\infty}\frac{Y(y)}{\eta(y)}\right)^+\right].
\]
In order to prove the equivalence of the statements (i) and (ii), we note that
\[
0\leq \left(\frac{Y(x)}{\eta(x)}- \limsup_{y\to\infty}\frac{Y(y)}{\eta(y)}\right)^+\leq  \frac{Y(x)}{\eta(x)}.
\]
Note also that $\bbE[Y(x)/\eta(x)]=1$ since  $Y(x)$ is independent of $1/\eta(x)\sim \mbox{Exp}(1)$. Using the fact that $(a-b)^+=a$ (for $a,b\geq 0$) if and only if $a=0$ or $b=0$, we can deduce that the equality
\[
\bbE\left[\left(\frac{Y(x)}{\eta(x)}- \limsup_{y\to\infty}\frac{Y(y)}{\eta(y)}\right)^+\right]=1
\]
occurs if and only if $\limsup_{y\to\infty}Y(y)/\eta(y)=0$\ a.e.\ on the event $\{Y(x) \neq  0\}$.
This proves the equivalence of (i) and (ii).
\end{proof}

\section{Proofs related to Section 3}

\subsection{Proofs of Propositions~\ref{prop1} and ~\ref{prop2}}
\begin{proof}[Proof of Proposition~\ref{prop1}]
According to Theorem \ref{thm:SKS}, $\eta$ is ergodic if and only if
\begin{equation}\label{eq:ergodic}
\lim_{r\to +\infty} \frac{1}{\lambda(B_r)}\int_{B_r}(2-\theta(h))\lambda(\mathrm{d}h)=0,
\end{equation}
and  $\eta$ is mixing if and only if
\begin{equation}\label{eq:mixing}
\lim_{h\to \infty}(2-\theta(h))=0.
\end{equation}
Clearly, in view of the inequalities \eqref{eq:comparison},  \eqref{eq:ergodic} is equivalent  to
\[
\lim_{r\to+\infty} \frac{1}{\lambda(B_r)}\int_{B_r}\bbP[h\in C(0)]\lambda(\mathrm{d}h)= \lim_{r\to+\infty}\bbE\Big[\frac{\lambda(C(0)\cap B_r)}{\lambda(B_r)}\Big]  = 0
\]
and  \eqref{eq:mixing} is equivalent  to $\lim_{h\to \infty}\bbP[h\in C(0)]=0$.
\end{proof}

\begin{proof}[Proof of Proposition~\ref{prop2}]
We use here an upper bound for the $\beta$-mixing coefficient provided by Dombry and Eyi-Minko \cite[Theorem 3.1]{DEM12}: the $\beta$-mixing coefficient $\beta(S_1,S_2)$ defined by  \eqref{eq:beta} satisfies
\[
\beta(S_1,S_2)\leq 2\bbP[A(S_1,S_2)],
\]
where $A(S_1,S_2)$ denotes the event
\[
\{\exists i\geq 1,\ \exists (s_1,s_2)\in S_1\times S_2,\ U_iY_i(s_1)=\eta(s_1) \ \mbox{and}\ U_iY_i(s_2)=\eta(s_2)\}.
\]
Introducing the cells $C(s_1)$ with $s_1\in S_1$, we have
\begin{eqnarray*}
A(S_1,S_2)&=&\{\exists (s_1,s_2)\in S_1\times S_2,\ s_2\in C(s_1)\}\\
&=& \{ \cup_{s_1\in S_1} C(s_1) \cap S_2 \neq \emptyset\}
\end{eqnarray*}
and
\[
\beta(S_1,S_2)\leq 2\bbP[ \cup_{s_1\in S_1} C(s_1) \cap S_2 \neq \emptyset].
\]
We deduce that, for all compact set $K\subset \mathcal{X}$ and for all $r>0$,
\[
\beta_r(K) \leq  2\bbP[\exists x\in \mathcal{X},\ d(x,K)\geq r \ \mbox{and}\ x\in \cup_{s\in K}C(s)].
\]
We prove below that if $C(0)$ is bounded a.s., then so is $\cup_{s\in K}C(s)$, whence the right-hand side in the above inequality converges to $0$ (by the dominated convergence theorem),
and $\lim_{r\to\infty} \beta_r(K)=0$.\\
Suppose now that $C(0)$ is bounded a.s. In the discrete case $\mathcal{X}=\bbZ^d$, the compact set $K$ is finite and $\cup_{s\in K} C(s)$ is a.s.\ bounded as a finite union of bounded sets.
In the continuous case $\mathcal{X}=\bbR^d$, $K$ may be infinite but it is known that there are a.s.\ only  finitely many indices $i\geq 1$ such that $U_iY_i(s)=\eta(s)$ for some $s\in K$ (see Dombry and Eyi-Minko \cite[Proposition 2.3]{DEM13}). Hence, we can extract a finite covering $\cup_{s\in K} C(s)=\cup_{j=1}^k C(s_j)$ and $\cup_{s\in K} C(s)$ is a.s.\ bounded as a finite union of bounded sets.
\end{proof}

\subsection{Proof of Lemma \ref{lem:dec}}
\begin{proof}[Proof of Lemma \ref{lem:dec}]
By the uniqueness of the max-L\'evy measure, the max-stable process $\eta$ is stationary if and only if its max-L\'evy measure $\mu$ is stationary.
By the properties of Poisson point processes, $\Phi\cap \mathcal{C}_i$, $i=1,2$, are independent Poisson point processes with intensity measures
$\mathrm{d}\mu_i=1_{\mathcal{C}_i}\mathrm{d}\mu$. The max-stable processes $\eta_1$ and $\eta_2$  are hence independent with exponent measures $\mu_1$ and  $\mu_2$, respectively.
Since the cone $\mathcal{C}_i$ is shift-invariant, so is the measure $\mu_i$. Hence, the process $\eta_i$ is stationary.
Finally, the distribution of $\eta_i$ is characterized by the max-L\'evy measure $\mathrm{d}\mu_i=1_{\mathcal{C}_i}\mathrm{d}\mu$ and does not depend on the representation \eqref{eq:deHaan}.
\end{proof}

\subsection{Proofs of Theorem~\ref{theo2} and Corollary~\ref{cor2}}
In the next lemma, we gather some preliminary computations needed for the proof of Theorem~\ref{theo2}.
\begin{lemma}\label{lem1} Let $x\in\cX$. We have:
\begin{itemize}
\item[i)] $\alpha_C=\bbE[ Y(x)1_{\{Y\in\cF_C\}}]$ and $\alpha_D=\bbE[ Y(x)1_{\{Y\in\cF_D\}}]$,
\item[ii)] $\bbP[\eta_C(x)>\eta_D(x)]=\alpha_C$ and $\bbP[\eta_D(x)>\eta_C(x)]=\alpha_D$,
\item[iii)] $\bbP[C(x)\ \mbox{is bounded},\ \eta_C(x)>\eta_D(x) ]=\bbE\left[ \left( \frac{Y(x)}{\eta(x)} - \limsup_{ \infty}\frac{Y}{\eta}\right)^+1_{\{Y\in\cF_C\}}\right]$,
\item[iv)] $\bbP[C(x)\ \mbox{is bounded},\  \eta_D(x)>\eta_C(x) ]=\bbE\left[ \left( \frac{Y(x)}{\eta(x)} - \limsup_{ \infty}\frac{Y}{\eta}\right)^+1_{\{Y\in\cF_D\}}\right]$.
\end{itemize}
\end{lemma}
\begin{proof}[Proof of Lemma~\ref{lem1}]
\textit{i)}
From \eqref{eq:cons} we get
\begin{eqnarray*}
\bbP[\eta_C(x)\leq y]
&=&\bbP[ \vee_{i\geq 1} U_iY_i(x)1_{\{Y_i\in\cF_C\}} \leq y ]\\
&=& \exp\Big(-\int_0^\infty \bbP[uY(x)1_{\{Y\in\cF_C\}}> y] u^{-2}\mathrm{d}u\Big)\\
&=& \exp(-\bbE[Y(x)1_{\{Y\in\cF_C\}}]/y),
\end{eqnarray*}
whence we deduce that $\alpha_C=\bbE[Y(x)1_{\{Y\in\cF_C\}}]$. The formula for $\alpha_D$ is obtained in the same way.

\vspace*{2mm}
\noindent
\textit{ii)} The random variables $\eta_C(x)$ and $\eta_D(x)$ are independent and have Fr\'echet distribution with parameters $\alpha_C$ and $\alpha_D$, respectively. Hence,
\begin{eqnarray*}
\bbP[\eta_C(x)>\eta_D(x)]
&=&\bbE[\exp(-\alpha_D/\eta_C(x))]\\
&=&\int_0^{+\infty} \exp(-\alpha_D/u)\mathrm{d}(\eee^{-\alpha_C/u})\\
&=&\alpha_C.
\end{eqnarray*}
For the last equality, we use $\alpha_C+\alpha_D=1$. Similarly,
$ \bbP[\eta_D(x)>\eta_C(x)]=\alpha_D$.

\vspace*{2mm}
\noindent
\textit{iii)} This statement is a variation of Corollary~\ref{cor1.2} and we give only the main lines of its proof.
We first prove the following version of  \eqref{eq:theo1.2}: for all compact sets $K\subset \mathcal{X}$,
\begin{eqnarray}
&&\bbP[C(x)\subset K,\ \eta_C(x)>\eta_D(x)]\nonumber\\
&=&\bbE\left[ \left( \frac{Y(x)}{\eta(x)} - \sup_{y\in K^c}\frac{Y(y)}{\eta(y)}\right)^+1_{\{Y\in\cF_C\}}\right].\label{eq:theo1.2bis}
\end{eqnarray}
Indeed, with the same notation as in the proof of  \eqref{eq:theo1.2}, we have
\begin{eqnarray*}
&&\{C(x)\subset K,\ \eta_C(x)>\eta_D(x)\}\\
&=&\{\exists i\geq 1,\ \phi_i(x)>m_i(x),\quad  \ \phi_i<_{K^c}m_i\quad \mbox{and}\ \phi_i\in\cF_C\}
\end{eqnarray*}
and the Slivnyak--Mecke formula entails that
\begin{eqnarray*}
&&\bbP[C(x)\subset K,\ \eta_C(x)>\eta_D(x)]\\
&=& \bbE\left[\sum_{i\geq 1}1_{\{\phi_i(x)>m_i(x)\}}1_{\{\phi_i<_{K^c}m_i\}}1_{\{\phi_i\in\cF_C\}}\right]\\
&=& \int_{\mathcal{F}_0}\bbE\left[1_{\{f(x)>\eta(x)\}}1_{\{f<_{K^c}\eta\}}1_{\{f\in\cF_C\}}\right]\mu(\mathrm{d}f).
\end{eqnarray*}
With similar computations as in the proof of  \eqref{eq:theo1.2},  \eqref{eq:theo1.2bis} is easily deduced. Then statement iii) follows from  \eqref{eq:theo1.2bis}
exactly in the same way as Corollary~\ref{cor1.2} follows from  \eqref{eq:theo1.2}.

\vspace*{2mm}
\noindent
\textit{iv)} The proof is similar and is omitted.
\end{proof}

\begin{proof}[Proof of Theorem~\ref{theo2}]
 Since $\{\eta_D(x)=\eta_C(x)\}$ is a null set, it suffices to prove the following two inclusions (modulo null sets):
\begin{eqnarray}
&&\{\eta_D(x)>\eta_C(x)\}\subset \{C(x)\ \mbox{is bounded}\}, \label{eq:proba1}\\
&&\{\eta_C(x)>\eta_D(x)\}\subset \{C(x)\ \mbox{is unbounded}\}. \label{eq:proba2}
\end{eqnarray}

\vspace*{2mm}
\noindent
\textit{Proof of  \eqref{eq:proba1}.}
We first reduce the proof of \eqref{eq:proba1} to the proof of
\begin{equation}\label{eq:proba1ter}
\lim_{y\to \infty}\frac{Y(y)}{\eta(y)}1_{\{Y\in\cF_D\}}=0\quad \mbox{a.s.}
\end{equation}
Indeed,  \eqref{eq:proba1ter} and statements i), ii) and iv) of Lemma~\ref{lem1} entail that
\begin{eqnarray*}
&&\bbP[C(x)\ \mbox{is bounded},\ \eta_D(x)>\eta_C(x)]\\
&=&\bbE\left[ \left( \frac{Y(x)}{\eta(x)} - \limsup_{\infty}\frac{Y}{\eta}\right)^+1_{\{Y\in\cF_D\}}\right]\\
&=& \bbE\left[ \frac{Y(x)}{\eta(x)} 1_{\{Y\in\cF_D\}}\right]\\
&=& \alpha_D\\
&=& \bbP[\eta_D(x)>\eta_C(x)],
\end{eqnarray*}
and we deduce  \eqref{eq:proba1}.

It remains to prove  \eqref{eq:proba1ter}.
Statements i) and iii) of Lemma \ref{lemBR}  imply that $Y1_{\{Y\in\cF_D\}}$ is Brown--Resnick stationary and such that
\[
\int_{\mathcal{X}} \sup_{y\in K} Y(x+y)1_{\{Y\in\cF_D\}}\lambda(\mathrm{d}x) < \infty \quad \mbox{a.s.}
\]
On the other hand, let us consider the process $Z=\frac{Y}{\eta}1_{\{Y\in\cF_D\}}$. Since $Y$ and  $1/\eta$ are Brown--Resnick stationary  and since the cone $\cF_D$ is shift invariant,
statement i) and ii) of Lemma \ref{lemBR} imply that $Z=\frac{Y}{\eta}1_{\{Y\in\cF_D\}}$ is Brown--Resnick stationary.
Furthermore, for any compact set $K\subset \cX$,
\begin{eqnarray*}
&&\bbE\left[\int_\mathcal{X} \sup_{y\in K}Z(x+y)\lambda(\mathrm{d}x)\ \Big|\ Y \right]\\
&\leq& \bbE\left[\int_\mathcal{X} \frac{\sup_{y\in K}Y(x+y)}{\inf_{y\in K}\eta(x+y)}1_{\{Y\in\cF_D\}}\lambda(\mathrm{d}x)\ \Big|\ Y \right]\\
&=& \bbE\left[\sup_{y\in K}\eta^{-1}(y) \right]   \int_{\mathcal{X}} \sup_{y\in K}Y(x+y)1_{\{Y\in\cF_D\}}\lambda(\mathrm{d}x)<\infty \quad \mbox{a.s.}
\end{eqnarray*}
In the last equation, we used the independence of $Y$ and $\eta$, the stationarity of $\eta$ and the fact that $\bbE\left[\sup_{y\in K}\eta^{-1}(y)\right]<\infty$ (see Dombry and Eyi-Minko \cite[Theorem 2.2]{DEM12}).
As a consequence,
\[
\int_\mathcal{X} \sup_{y\in K}Z(x+y)\lambda(\mathrm{d}x)<\infty\quad \mbox{a.s.}
\]
and Lemma \ref{lemBR}-iii) implies that $\lim_{x\to\infty} Z(x)=0$ a.s., thus proving \eqref{eq:proba1ter}.

\vspace*{2mm}
\noindent
\textit{Proof of  \eqref{eq:proba2}.}
We consider the shift-invariant cone
\[
\cF_L=\left\{f\in\cF_0; \sup_{\cX} f>\limsup_\infty f\right\}.
\]
We will  prove that the process $Z=\frac{Y}{\eta}1_{\{Y\in\cF_C\}}$ is Brown--Resnick stationary and satisfies
\begin{equation}\label{eq:proba2bis}
\bbP[Z\in\cF_L ]=0.
\end{equation}
After this has been done,  \eqref{eq:proba2} can be deduced as follows:  \eqref{eq:proba2bis} implies that
\[
\frac{Y(x)}{\eta(x)}1_{\{Y\in\cF_C\}}\leq \sup_\mathcal{X} \left(\frac{Y}{\eta}1_{\{Y\in\cF_C\}}\right) \leq\left( \limsup_{\infty}\frac{Y}{\eta}1_{\{Y\in\cF_C\}}\right)\quad \mathrm{a.s.},
\]
whence
\[
\left( \frac{Y(x)}{\eta(x)} - \limsup_{\infty}\frac{Y}{\eta}\right)^+1_{\{Y\in\cF_C\}}=0\quad \mathrm{a.s.}
\]
According to Lemma~\ref{lem1}, statement iii), we obtain that
\begin{eqnarray*}
&&\bbP[C(x)\ \mbox{is bounded},\ \eta_C(x)>\eta_D(x) ]\\
&=&\bbE\left[\left( \frac{Y(x)}{\eta(x)} - \limsup_{\infty}\frac{Y}{\eta}\right)^+1_{\{Y\in\cF_C\}}\right]\\
&=&0,
\end{eqnarray*}
and this implies  \eqref{eq:proba2}.

We now consider  \eqref{eq:proba2bis}. Statements i) and ii) of Lemma \ref{lemBR} imply that the process $Z$ is Brown--Resnick stationary.
Lemma \ref{lemBR}-iv) entails that
$\bbP[Z\in\cF_L]\leq \bbP[ Z\in\cF_D]$.
So, it suffices to prove that $\bbP[Z\in\cF_D]=0$. Suppose by contradiction that $\bbP[Z\in\cF_D]>0$. Recalling that $Z=\frac{Y}{\eta}1_{\{Y\in\cF_C\}}$, we see that
\[
\{Z\in \cF_D\}=\{Y\in\cF_C\}\cap \{Y/\eta\in \cF_D\}.
\]
On  the set $\{Y\in \cF_C\}=\{\limsup_\infty Y>0\}$, one can construct a $\sigma(Y)$-measurable random sequence $x_n\to \infty$ such that
$Y(x_n)\geq \frac{1}{2}\limsup_\infty Y>0$.
Then, on  $\{Z\in \cF_D\}\subset \{Y/\eta\in \cF_D\}=\{\lim_\infty Y/\eta=0\}$, we have necessarily $\eta(x_n)\to +\infty$. But $\eta$ is stationary and independent of $Y$, so that $\eta(x_n)$
has a unit Fr\'echet distribution that does not depend on $n$. This leads to a contradiction and  we must hence have  $\bbP[Z\in\cF_D]=0$. This concludes the proof of
\eqref{eq:proba2bis}.
\end{proof}

\begin{proof}[Proof of Corollary \ref{cor2}]
Theorem \ref{theo2} and Lemma \ref{lem1}-ii) together yield
\[
 \bbP[C(x)\ \mbox{is unbounded}]=\bbP[\eta_C(x)>\eta_D(x)]=\alpha_C,
\]
proving statement i). Statement ii) is proved similarly. Furthermore, $\eta$ is purely dissipative if $\eta_C=0$, which is  equivalent to $\alpha_C=0$.
We deduce easily that  $\eta$ is purely dissipative if and only if $C(x)$ is bounded a.s.\ and this proves  iii).
The proof of  iv) is similar.
\end{proof}

\subsection{Proofs of Theorem~\ref{theo3} and Corollary~\ref{cor3}}
\begin{proof}[Proof of Theorem~\ref{theo3}]
It suffices to prove the following two inclusions (modulo null sets):
\begin{equation}\label{eq:inclusion1bis}
\{\eta_N(x)>\eta_P(x)\}\subset \{ \delta^-(C(x))=0\}
\end{equation}
and
\begin{equation}\label{eq:inclusion2bis}
\{\eta_P(x)>\eta_N(x)\}\subset\{\delta(C(x))>0\}.
\end{equation}

\vspace*{2mm}
\noindent
\textit{Proof of \eqref{eq:inclusion1bis}.}
Let us consider the  cell of $x$ with respect to the null component only. It is defined by
\[
C_N(x)=\{y\in \mathcal{X};\ \exists i\geq 1,\ Y_i\in\cF_N,\ U_iY_i(x)=\eta_N(x),\  U_iY_i(y)=\eta_N(y)\}.
\]
Clearly, $\eta_N(x)>\eta_P(x)$ implies that $C(x)\subset C_N(x)$. We will prove that $\delta^-(C_N(x))=0$ on $\{\eta_N(x)>\eta_P(x)\}$ and this
implies  \eqref{eq:inclusion1bis}.

We can suppose without loss of generality that $\eta=\eta_N$ is generated by a null flow and prove that the lower asymptotic density of $C(x)=C_N(x)$ is equal to zero.
According to Wang {\it et al.} \cite[Theorem~4.1]{WRS13} or Kabluchko~\cite[Theorem~8]{K09}, max-stable random fields generated by null flows are ergodic, whence Proposition~\ref{prop1} implies
\[
\bbE\Big[\frac{\lambda(C(0)\cap B_r)}{\lambda(B_r)}\Big]\to 0\quad \mbox{as}\ r\to\infty.
\]
This  implies the convergence in probability
\[
\frac{\lambda(C(0)\cap B_r)}{\lambda(B_r)}\stackrel{\bbP}\longrightarrow 0,\quad \mbox{as}\ r\to +\infty
\]
and hence almost sure converge to $0$ along a subsequence. We deduce that $\delta^-(C(0))=0$ almost surely and, by stationarity, the same holds true for $C(x)$, $x\in \mathcal{X}$.

\vspace*{2mm}
\noindent
\textit{Proof of Equation \eqref{eq:inclusion2bis}.}
Possibly changing representation \eqref{eq:deHaan}, we may suppose without loss of generality that for any $i\geq 1$, the random process
$\tilde Y_i=Y_i1_{\{Y_i\in P\}}$ is stationary.  We consider the cells
\[
\tilde C_i=\{y\in \mathcal{X},\ U_i\tilde Y_i(y)=\eta(y)\}, \quad i\geq 1.
\]
We will prove below that for every $i\geq 1$ with probability one,
\begin{equation}\label{eq:positivecell}
\mbox{either } \delta(\tilde C_i)>0 \mbox{ or } \lambda(\tilde C_i)=0.
\end{equation}
We show that this implies \eqref{eq:inclusion2bis}. On the event $\{\eta_P(x)>\eta_N(x)\}$, there is a random index $i(x)$ such that $C(x)=\tilde C_{i(x)}$. Furthermore, since $x\in C(x)$, we have $\lambda(\tilde C_{i(x)})>0$ (this is clear in the discrete case, in the continuous case, $C(x)$ contains a neighborhood of $x$). According to \eqref{eq:positivecell}, we obtain $\delta(C_{i(x)})=\delta(C_{x})>0$, proving   \eqref{eq:inclusion2bis}.

It remains to prove  \eqref{eq:positivecell}.
Recall that the $U_i$'s are arranged in the decreasing order. Fix $i\geq 1$  and observe that the distribution of $(U_i,\tilde Y_i,\eta)$ is invariant under the shift
\[
T_x(u,f_1,f_2)=(u,f_1(\cdot+x),f_2(\cdot+x)),\quad u>0,\ f_1,f_2\in\cF_0.
\]
Then we observe that
\begin{eqnarray*}
\frac{\lambda(\tilde C_i\cap B_r)}{\lambda(B_r)}&=&\frac{1}{\lambda(B_r)}\int_{B_r}1_{\{x\in \tilde C_i\}}\lambda(\mathrm{d}x)\\
&=&\frac{1}{\lambda(B_r)}\int_{B_r}1_{\{U_i\tilde Y_i(x)=\eta(x)\}}\lambda(\mathrm{d}x)\\
&=&\frac{1}{\lambda(B_r)}\int_{B_r}1_{\{T_x(U_i,\tilde Y_i,\eta)\in A\}}\lambda(\mathrm{d}x)
\end{eqnarray*}
with $A=\{(u,f_1,f_2); uf_1(0)=f_2(0)\}$. We can then apply the multiparameter ergodic theorem (see, e.g., \cite[Theorem 2.8]{WRS13}) and conclude that
\[
\lim_{r\to +\infty} \frac{\lambda(\tilde C_i\cap B_r)}{\lambda(B_r)}=\bbE[1_A(U_i,\tilde Y_i,\eta)\mid \cI] \quad \mbox{a.s.},
\]
where $\cI$ denotes the  $\sigma$-algebra of shift-invariant sets. This shows that $\tilde C_i$ has an  asymptotic density,
\[
\delta(\tilde C_i)=\bbE[1_{\{0\in\tilde C_i\}}\mid \cI]\quad \mbox{a.s.}
\]
Furthermore, we observe that  shift-invariance implies that
\[
\bbE[1_{\{0\in\tilde C_i\}}\mid \cI]=\bbE[1_{\{x\in\tilde C_i\}}\mid \cI],\quad x\in \mathcal{X}.
\]
Using the fact that $\{\delta(\tilde C_i)=0\}\in \cI$, we deduce that
\begin{eqnarray*}
\bbE[\lambda(\tilde C_i) 1_{\{\delta(\tilde C_i)=0\}}\mid \cI]&=& 1_{\{\delta(\tilde C_i)=0\}} \int_{\mathcal{X}}\bbE[1_{\{x\in \tilde C_i\}}\mid \cI]\, \lambda(\mathrm{d}x)\\
&=& 0.
\end{eqnarray*}
Taking the expectation, we obtain that
\[
\bbE[\lambda(\tilde C_i) 1_{\{\delta(\tilde C_i)=0\}}]=0
\]
and we conclude that $\lambda(\tilde C_i)=0$ on the event $\{\delta(\tilde C_i)=0\}$, proving \eqref{eq:positivecell}.
\end{proof}

\begin{proof}[Proof of Corollary~\ref{cor3}]
For the sake of brevity, we omit the proof which is quite straightforward from Theorem~\ref{theo3} and very similar to the proof of Corollary~\ref{cor2}.
\end{proof}

\section*{Acknowledgements}
The authors are grateful to the editor, the associate editor and two referees for their helpful suggestions.

\appendix
\section{Proof of Lemma \ref{lemBR} iv)}
\begin{proof}[Proof of Lemma \ref{lemBR} iv)]
To check that $\cF_L$ is localizable, take $L_1(f)=\argmax(f)$ and $L_2(f)=\max(f)$ in the definition of a localizable cone (note that we are working with continuous functions so that the supremum is a maximum).

For the proof of the inclusion $\{Y\in\mathcal{F}_L\}\subset \{\lim_{x\to \infty} Y(x)=0\}$,
we prove that $\eta_L=\vee_{i\geq 1} U_iY_i 1_{\{Y_i\in \mathcal{F}_L\}}$ admits a mixed moving maximum representation.
According to \cite[Proposition 10]{DK15}, this implies that $Y_L\in \tilde{\mathcal{F}_D}$ almost surely and hence the inclusion
$\{Y\in\mathcal{F}_L\}\subset \{\lim_{x\to\infty} Y(x)=0\}$ modulo null sets.
For simplicity, we omit the subscript $L$ and assume that $Y\in\mathcal{F}_L$ almost surely. We prove that $\eta=\vee_{i\geq 1} U_iY_i$ admits a mixed moving maximum representation. In fact, the proof works if $\cF_L$ is replaced by any localizable cone.  We follow the proof of Theorem 14 in Kabluchko et al.\ \cite{KSdH09} and we sketch only the main lines. We  introduce the random variables
\[
 X_i=\mathop{\mathrm{argmax}}_{x\in\cX}Y_i(x),\quad Z_i(\cdot)=\frac{Y_i(X_i+\cdot)}{\max_{x\in\cX} Y_i (x)},\quad V_i=U_i\max_{x\in\cX} Y_i (x).
\]
Note that $X_i$ is well-defined because of the definition of $\cF_L$. If the maximum is attained at several points, we take the lexicographically smallest one. Clearly, we have $U_iY_i(x)=V_iZ_i(x-X_i)$ for all $x\in\cX$ so that
\[
\eta(x)=\bigvee_{i\geq 1}V_iZ_i(x-X_i).
\]
It remains to check that $(V_i,X_i,Z_i)_{i\geq 1}$  is a Poisson point process with  intensity measure
$u^{-2}\mathrm{d}u\lambda(\mathrm{d}x)Q(\mathrm{d}f)$, where $Q$ is a probability measure on $\cF_0$.
 Clearly, $(V_i,X_i,Z_i)_{i\geq 1}$ is a Poisson point process as the image of the original point process $(U_i,Y_i)_{i\geq 1}$. Its intensity  is the  image of the  intensity of the original point process. With a straightforward transposition of the arguments of \cite[Theorem 14]{KSdH09}, one can check that it has the required form.
\end{proof}

\bibliographystyle{plain}
\bibliography{Biblio}

\end{document}